\documentclass{my_aims}

\usepackage{amsmath}
\usepackage{paralist}

\usepackage[colorlinks=true]{hyperref}
\hypersetup{urlcolor=blue, citecolor=red}

\def\ppages{X--XX}

\newtheorem{theorem}{Theorem}[section]
\newtheorem{corollary}{Corollary}
\newtheorem{lemma}[theorem]{Lemma}
\theoremstyle{definition}
\newtheorem{definition}[theorem]{Definition}

\renewcommand{\subjclassname}{\textup{2010} Mathematics Subject Classification}

\title[Infinite horizon variational problems on time scales]{Necessary optimality conditions
for infinite horizon variational problems on time scales}

\author[M. Dryl and D. F. M. Torres]{}



\email{monikadryl@ua.pt}
\email{delfim@ua.pt}

\thanks{Part of first author's Ph.D., which is carried out
at the University of Aveiro under the
\emph{Doctoral Programme in Mathematics and Applications}
of Universities of Aveiro and Minho.}

\begin{document}

\maketitle

\centerline{\scshape Monika Dryl}
\medskip
{\footnotesize
 \centerline{Center for Research and Development in Mathematics and Applications}
 \centerline{Department of Mathematics, University of Aveiro, 3810-193 Aveiro, Portugal}
}

\medskip

\centerline{\scshape Delfim F. M. Torres}
\medskip
{\footnotesize
 \centerline{Center for Research and Development in Mathematics and Applications}
 \centerline{Department of Mathematics, University of Aveiro, 3810-193 Aveiro, Portugal}}

% --------------------------------

\bigskip

\centerline{To George Leitmann on the occasion of his 87th birthday.}

% --------------------------------

\begin{abstract}
We prove Euler--Lagrange type equations and transversality conditions
for generalized infinite horizon problems of the calculus of variations
on time scales. Here the Lagrangian depends on the independent
variable, an unknown function and its nabla derivative,
as well as a nabla indefinite integral
that depends on the unknown function.
\end{abstract}

% --------------------------------

\section{Introduction}

In recent years, it has been shown that the behavior
of many systems is described more accurately
using dynamic equations on a time scale or a measure chain
\cite{BohnerDEOTS,MBbook2001,LM:ts}.
If the variational principle holds as a unified law \cite{book:Leitmann},
then the above time scale differential equations must also
come from minimization of some delta or nabla functional
with a Lagrangian containing delta or nabla derivative terms
\cite{MyID:138,comNaty:AML,naty:irlanda}.
Here we consider the following infinite horizon variational problem:
\begin{equation}
\label{rown z introduction}
\mathcal{J}(x)=\int\limits_{a}^{\infty}
L\left(t,x^{\rho}(t),x^{\nabla}(t),z(t)\right)\nabla t \longrightarrow \textrm{extr},
\end{equation}
where ``extr'' means ``minimize'' or ``maximize''. The variable $z(t)$ is defined by
\begin{equation*}
z(t)=\int\limits_{a}^{t} g\left(\tau, x^{\rho}(\tau),x^{\nabla}(\tau)\right)\nabla\tau.
\end{equation*}
Integral \eqref{rown z introduction} does not necessarily converge, being possible
to diverge to plus or minus infinity or oscillate.
Problem \eqref{rown z introduction} generalizes the ones
recently studied in \cite{AM-NM-DT,Generalizing_the_variational_theory}.

The paper is organized as follows. In Section~\ref{sec:prelim}
we collect the necessary definitions and results
of the nabla calculus on time scales, which are necessary in the sequel.
In Section~\ref{sec:mr} we state and prove the new results:
we prove necessary optimality conditions to problem \eqref{rown z introduction},
obtaining Euler--Lagrange type equations in the class of functions
$x\in C_{ld}^{1}\left(\mathbb{T},\mathbb{R}^n\right)$
and new transversality conditions
(Theorems~\ref{maintheorem} and \ref{secondmaintheorem}).

% -------------------------------------

\section{Preliminaries}
\label{sec:prelim}

In this section we introduce basic definitions
and theorems that are needed in Section~\ref{sec:mr}.
For more on the time scale theory we refer the reader
to \cite{BohnerDEOTS,MBbook2001,TorresDeltaNabla}.
A time scale $\mathbb{T}$ is an arbitrary nonempty closed subset of $\mathbb{R}$.
All the intervals in this paper are time scale intervals
with respect to a given time scale $\mathbb{T}$
(for example, by $[a,b]$ we mean $[a,b] \cap \mathbb{T}$).

\begin{definition}[\textrm{e.g.}, Section 2.1 of \cite{TorresDeltaNabla}]
The forward jump operator $\sigma:\mathbb{T} \rightarrow \mathbb{T}$
is defined by $\sigma(t):=\inf\{s \in \mathbb{T}:s>t\}$ for $t\neq\sup \mathbb{T}$
and $\sigma(\sup\mathbb{T})=\sup\mathbb{T}$ if $\sup\mathbb{T}<+\infty$.
Similarly, the backward jump operator $\rho:\mathbb{T}\rightarrow \mathbb{T}$
is defined by $\rho(t):=\sup\lbrace s\in\mathbb{T}: s<t\rbrace$
for $t\neq \inf\mathbb{T}$ and $\rho(\inf\mathbb{T})=\inf\mathbb{T}$
if $\inf\mathbb{T}>-\infty$.
\end{definition}

\begin{definition}[\textrm{e.g.}, Section 2.1 of \cite{TorresDeltaNabla}]
The backward graininess function
$\nu:\mathbb{T} \rightarrow [0,\infty)$
is defined by $\nu(t):=t-\rho(t)$.
\end{definition}

A point $t \in \mathbb{T}$ is called right-dense, right-scattered,
left-dense or left-scattered if $\sigma(t) = t$, $\sigma(t) > t$,
$\rho(t) = t$, $\rho(t) < t$, respectively. We say that $t$ is isolated if
$\rho(t) < t < \sigma(t)$, that $t$ is dense if $\rho(t) = t = \sigma(t)$.
If $\mathbb{T}$ has a right-scattered minimum $m$, then
define $\mathbb{T}_{\kappa}:=\mathbb{T}-\lbrace m\rbrace$;
otherwise, set $\mathbb{T}_{\kappa}:=\mathbb{T}$.
To simplify the notation, let $f^{\rho}(t):=f(\rho(t))$.

\begin{definition}[\textrm{e.g.}, Section 2.2 of \cite{TorresDeltaNabla}]
We say that function $f:\mathbb{T}\rightarrow\mathbb{R}$
is nabla differentiable at $t\in\mathbb{T}_{\kappa}$ if there is
a number $f^{\nabla}(t)$ such that for all $\epsilon >0$ there exists
a neighborhood $U$ of $t$ such that
\begin{equation*}
|f^{\rho}(t)-f(s)-f^{\nabla}(t)(\rho(t)-s)|
\leq \epsilon |\rho(t)-s| \mbox{ for all } s \in U.
\end{equation*}
We call $f^{\nabla}(t)$ the nabla derivative of $f$ at $t$. Moreover,
$f$ is nabla differentiable on $\mathbb{T}$ provided $f^{\nabla}(t)$
exists for all $t\in\mathbb{T}_{\kappa}$.
\end{definition}

\begin{theorem}[\textrm{e.g.}, Theorem 8.41 of \cite{BohnerDEOTS}]
\label{tw:differprop}
Assume $f,g:\mathbb{T}\rightarrow\mathbb{R}$
are nabla differentiable at $t\in\mathbb{T_{\kappa}}$. Then:
\begin{enumerate}

\item The sum $f+g:\mathbb{T}\rightarrow\mathbb{R}$
is nabla differentiable at $t$ with
\begin{equation*}
(f+g)^{\nabla}(t)=f^{\nabla}(t)+g^{\nabla}(t).
\end{equation*}

\item For any constant $\alpha$, $\alpha f:\mathbb{T}\rightarrow\mathbb{R}$
is nabla differentiable at $t$ and
\begin{equation*}
(\alpha f)^{\nabla}(t)=\alpha f^{\nabla}(t).
\end{equation*}

\item The product $fg:\mathbb{T}\rightarrow\mathbb{R}$
is nabla differentiable at $t$ and the following product rules hold:
\begin{equation*}
(fg)^{\nabla}(t)=f^{\nabla}(t)g(t)+f^{\rho}g^{\nabla}(t)
= f^{\nabla}(t)g^{\rho}(t) + f(t)g^{\nabla}(t).
\end{equation*}

\item If $g(t)g^{\rho}(t)\neq 0$, then $f/g$
is nabla differentiable at $t$ and the following quotient rule hold:
\begin{equation*}
\bigg(\frac{f}{g}\bigg)^{\nabla}(t)
=\frac{f^{\nabla}(t)g(t)-f(t)g^{\nabla}(t)}{g(t)g^{\rho}(t)}.
\end{equation*}
\end{enumerate}
\end{theorem}

\begin{definition}[\textrm{e.g.}, Section~3.1 of \cite{MBbook2001}]
Let $\mathbb{T}$ be a time scale, $f:\mathbb{T} \rightarrow \mathbb{R}$.
We say that function $f$ is ld-continuous if it is continuous at left-dense
points and its right-sided limits exist (finite) at all right-dense points.
\end{definition}

\begin{definition}[\textrm{e.g.}, Definition~9 of \cite{TorresDeltaNabla}]
A function $F : \mathbb{T} \rightarrow \mathbb{R}$
is called a nabla antiderivative of $f : \mathbb{T} \rightarrow \mathbb{R}$
provided $F^\nabla(t) = f(t)$ for all $t \in \mathbb{T}_\kappa$.
In this case we define the nabla integral of $f$ from $a$ to $b$
($a, b \in \mathbb{T}$) by
$$
\int_a^b f(t) \nabla t := F(b) - F(a).
$$
\end{definition}

The set of all ld-continuous functions $f:\mathbb{T}\rightarrow \mathbb{R}$
is denoted by $C_{ld}=C_{ld}(\mathbb{T},\mathbb{R})$,
and the set of all nabla differentiable functions with ld-continuous derivative by
$C^{1}_{ld}=C^{1}_{ld}(\mathbb{T},\mathbb{R})$.

\begin{theorem}[\textrm{e.g.}, Theorem~8.45 of \cite{BohnerDEOTS}
or Theorem~11 of \cite{TorresDeltaNabla}]
Every ld-continuous function $f$ has a nabla antiderivative $F$.
In particular, if $a \in \mathbb{T}$, then $F$ defined by
$$
F(t) = \int_a^t f(\tau) \nabla \tau, \quad t \in \mathbb{T},
$$
is a nabla antiderivative of $f$.
\end{theorem}

\begin{theorem}[\textrm{e.g.}, Theorem~8.47 of \cite{BohnerDEOTS}
or Theorem~12 of \cite{TorresDeltaNabla}]
\label{tw:intprop}
If $a,b\in\mathbb{T}$, $a\leq b$, and $f,g\in\ C_{ld}(\mathbb{T}, \mathbb{R})$, then
\begin{enumerate}

\item $\displaystyle \int\limits_{a}^{b}(f(t)+g(t))\nabla t
=\int\limits_{a}^{b} f(t)\nabla t+\int\limits_{a}^{b}g(t)\nabla t$;

\item $\displaystyle \int\limits_{a}^{a} f(t)\nabla t=0$;

\item $\displaystyle \int\limits_{a}^{b}f(t)g^{\nabla}(t)\nabla t
=\left.f(t)g(t)\right|^{t=b}_{t=a}-\int\limits_{a}^{b}f^{\nabla}(t)g^\rho(t)\nabla t$;

\item If $f(t)>0$ for all $a<t\leq b$,
then $\int\limits_{a}^{b}f(t)\nabla t >0$;

\item If $t\in\mathbb{T}_{\kappa}$,
then $\displaystyle \int_{\rho(t)}^{t}f(\tau)\nabla \tau=\nu(t)f(t)$.
\end{enumerate}
\end{theorem}

\begin{definition}
If $a\in\mathbb{T}$, $\sup\mathbb{T}=+\infty$
and $f\in C_{ld}([a,+\infty[,\mathbb{R})$,
then we define the improper nabla integral by
\begin{equation*}
\int\limits_{a}^{+\infty} f(t)\nabla t
:= \lim\limits_{b\to +\infty}\int\limits_{a}^{b} f(t)\nabla t,
\end{equation*}
provided this limit exists
(in $\overline {\mathbb{R}}:=\mathbb{R}\cup\lbrace -\infty, +\infty\rbrace$).
\end{definition}

\begin{theorem}[\textrm{e.g.}, \cite{LangUnderAnalysis}]
Let $S$ and $T$ be subsets of a normed vector space.
Let $f$ be a map defined on $T\times S$,
having values in some complete normed vector space.
Let $v$ be adherent to $S$ and $w$ adherent to $T$. Assume that
\begin{enumerate}
\item $\lim\limits_{x\to v} f(t,x)$ exists for each $t\in T$;
\item $\lim\limits_{t\to w} f(t,x)$ exists uniformly for $x\in S$.
\end{enumerate}
Then, $\lim\limits_{t\to w}\lim\limits_{x\to v} f(t,x)$,
$\lim\limits_{x\to v}\lim\limits_{t\to w} f(t,x)$ and
$\lim\limits_{(t,x)\to (w,v)} f(t,x)$ all exist and are equal.
\end{theorem}

The next result can be easily obtained from Theorem~4 of
\cite{Generalizing_the_variational_theory} by using the
delta-nabla duality theory of time scales
\cite{cc:dual,comBasia:duality,comNaty:AML}.

\begin{theorem}
\label{tw:necessary optimality condition}
Suppose that $x_{\star}$ is a local minimizer or local maximizer to problem
\begin{equation*}
\mathcal{L}(x) = \int\limits_{a}^{b}
L\left(t,x^{\rho}(t),x^{\nabla}(t),z(t)\right)\nabla t
\longrightarrow \textrm{extr},
\end{equation*}
where the variable $z$ is the integral defined by
\begin{equation*}
z(t)=\int\limits_{a}^{t}
g\left(\tau, x^{\rho}(\tau),x^{\nabla}(\tau)\right)\nabla\tau,
\end{equation*}
in the class of functions $x\in C^{1}_{ld}(\mathbb{T}, \mathbb{R}^{n})$
satisfying the boundary conditions
$x(a)=\alpha$ and $x(b)=\beta$.
Then, $x_{\star}$ satisfies the Euler--Lagrange system of equations
\begin{multline*}
g_{x}\langle x\rangle (t)
\int\limits_{\rho(t)}^{b} L_{z}[x,z](\tau)\nabla\tau
-\left(g_{v}\langle x\rangle(t)
\int\limits_{\rho(t)}^{b}L_{z}[x,z](\tau)\nabla \tau\right)^{\nabla}\\
+ L_{x}[x,z](t) -L_{v}^{\nabla}[x,z](t) =0
\end{multline*}
for all $t\in [a,b]_{\kappa}$, where
$L_{x}$, $L_{v}$ and $L_{z}$ are,
respectively, the partial derivatives of $L(\cdot, \cdot, \cdot, \cdot)$
with respect to its second, third and fourth argument,
$g_{x}$ and $g_{v}$ are, respectively,
the partial derivatives of $g(\cdot,\cdot,\cdot)$ with respect
to its second and third argument, and the operators $[\cdot,\cdot]$
and $\langle \cdot \rangle$ are defined by
$[x,z](t):= \left(t, x^{\rho}(t), x^{\nabla}(t), z(t)\right)$
and $\langle x\rangle(t):=\left(t, x^{\rho}(t), x^{\nabla}(t)\right)$.
\end{theorem}

% -------------------------------------

\section{Main Results}
\label{sec:mr}

Let $\mathbb{T}$ be a time scale such that $\sup \mathbb{T} =+\infty$.
Suppose that $a$, $T$, $T'\in\mathbb{T}$ are such that $T>a$ and $T'>a$.
The meaning of $L[x,z](t)$, $g\langle x\rangle(t)$ and that of
partial derivatives $L_{x}[x,z](t)$, $L_{v}[x,z](t)$,
$L_{z}[x,z](t)$, $g_{x}\langle x\rangle(t)$
and $g_{v}\langle x\rangle(t)$ is given
in Theorem~\ref{tw:necessary optimality condition}.
Let us consider the following variational problem on $\mathbb{T}$:
\begin{equation}
\label{eq:main result}
\mathcal{J}(x) :=
\int_a^\infty L[x,z](t) \nabla t =
\int_a^\infty L\left(t,x^{\rho}(t),x^{\nabla}(t),z(t)\right)\nabla t
\longrightarrow \max
\end{equation}
subject to $x(a)=x_{a}$. The variable $z$ is the integral defined by
\begin{displaymath}
z(t) := \int_{a}^{t}g\langle x\rangle(\tau)\nabla\tau
= \int_{a}^{t}g\left(\tau, x^{\rho}(\tau),x^{\nabla}(\tau)\right)\nabla\tau.
\end{displaymath}
We assume that $x_{a}\in\mathbb{R}^{n}$, $n\in\mathbb{N}$,
$(u,v,w)\rightarrow L(t,u,v,w)$ is a
$C_{ld}^{1}(\mathbb{R}^{2n+1},\mathbb{R})$
and $(u,v)\rightarrow g(t,u,v)$ a
$C_{ld}^{1}(\mathbb{R}^{2n},\mathbb{R})$
function for any $t\in\mathbb{T}$, and functions
$t \rightarrow L_v[x,z](t)$
and $t \rightarrow g_v\langle x\rangle(t)$
are nabla differentiable for all
$x\in C_{ld}^{1}(\mathbb{T},\mathbb{R}^{n})$.

\begin{definition}
\label{df:1}
We say that $x$ is an admissible path for problem \eqref{eq:main result}
if $x\in C^{1}_{ld}\left(\mathbb{T}, \mathbb{R}^{n}\right)$ and $x(a)=x_{a}$.
\end{definition}

\begin{definition}
\label{df:2}
We say that $x_{\star}$ is a weak maximizer to problem \eqref{eq:main result}
if $x_{\star}$ is an admissible path and, moreover,
$$
\lim\limits_{T\to +\infty}\inf\limits_{T'\geq T}
\int\limits_{a}^{T'} \left(L[x,z](t)
-L[x_{\star},z_{\star}](t)\right)\nabla t \leq 0
$$
for all admissible path $x$.
\end{definition}

\begin{lemma}
\label{lem:Dubois-Reymond rho}
Let $g\in C_{ld}(\mathbb{T}, \mathbb {R})$. Then,
$$
\lim\limits_{T\to\infty} \inf\limits_{T'\geq T}\int\limits_{a}^{T'}
g(t)\eta^{\rho}(t)\nabla t=0
$$
for all $\eta\in C_{ld}\left(\mathbb{T}, \mathbb {R}\right)$
such that $\eta (a)=0$ if, and only if, $g(t)=0$ on $[a, +\infty[$.
\end{lemma}

\begin{proof}
The implication $\Leftarrow$ is obvious.
Let us prove the implication $\Rightarrow$
by contradiction. Suppose that $g(t)\not\equiv 0$.
Let $t_{0}$ be a point on $[a,+\infty[$ such that $g(t_{0})\neq 0$.
Suppose, without loss of generality, that $g(t_{0})> 0$.
Two situations may occur: $t_{0}$ is left-dense (case I)
or $t_{0}$ is left-scattered (case II).
Case I: if $t_{0}$ is left-dense, then function $g$
is positive on $[t_{1}, t_{0}]$ for $t_{1}<t_{0}$. Define:
\begin{equation*}
\eta(t)=\left\{
\begin{array}{l}
(t_{0}-t)(t-t_{1})\ \mbox{for } t\in [t_{1},t_{0}],
\\0\ \mbox{ otherwise.}
\end{array}\right.
\end{equation*}
Then,
\begin{equation*}
\eta(\rho(t))=\left\{
\begin{array}{l}
\left(t_{0}-\rho(t)\right)\left(\rho(t)-t_{1}\right)
\mbox{ for } \rho(t)\in [t_{1},t_{0}],\\
0\ \mbox{ otherwise}.
\end{array}\right.
\end{equation*}
If $\rho(t) \in [t_{1},t_{0}]$,
then $\eta(\rho(t))=(t_{0}-\rho(t))(\rho(t)-t_{1})>0$. Thus,
$$
\lim\limits_{T\to +\infty} \inf\limits_{T'\geq T}
\int\limits_{a}^{T'}g(t)\eta^{\rho}(t)\nabla t
=\int\limits_{t_{1}}^{t_{0}}g(t)\eta(\rho(t))\nabla t >0
$$
and we obtain a contradiction. Case II: $t_{0}$ is left-scattered.
Two situations are then possible: $\rho(t_{0})$ is left-scattered
or $\rho(t_{0})$ is left-dense.
If $\rho(t_{0})$ is left-scattered, then
$\rho(\rho(t_{0}))<\rho(t_{0})<t_{0}$.
Let $t\in [\rho(t_{0}), t_{0}]$. Define
$$
\eta(t)=\left\{
\begin{array}{l}
g(t_{0})\ \mbox{ for } t=\rho(t_{0}),\\
0\ \mbox{otherwise.}
\end{array}\right.
$$
Then,
$$
\eta(\rho(t_{0}))
=\left\{
\begin{array}{l}
g(t_{0})\ \mbox{for } \rho(t_{0})=\rho(\rho(t_{0})),\\
0\ \mbox{otherwise.}
\end{array}\right.
$$
It means that $\eta(\rho(t_{0}))=g(t_{0})>0$.
From point 5 of Theorem~\ref{tw:intprop}, we obtain
\begin{equation*}
\begin{split}
\lim\limits_{T\to +\infty} \inf\limits_{T'\geq T}\int\limits_{a}^{T'}g(t)\eta^{\rho}(t)\nabla t
&= \int\limits_{\rho(t_{0})}^{t_{0}} g(t)\eta^{\rho}(t)\nabla t\\
&= g(t_{0})\eta(\rho(t_{0}))\nu(t_{0})=g(t_{0})g(t_{0})(t_{0}-\rho(t_{0}))>0,
\end{split}
\end{equation*}
which is a contradiction. It remains to consider
the situation when $\rho(t_{0})$ is left-dense.
Two cases are then possible: $g(\rho(t_{0}))\neq 0$ or $g(\rho(t_{0})) =0$.
If $g(\rho(t_{0}))\neq 0$, then we can assume that $g(\rho(t_{0}))>0$
and $g$ is also positive in $[t_{2},\rho(t_{0})]$ for $t_{2}<\rho(t_{0})$.
Define
$$
\eta(t)=\left\{
\begin{array}{l}
\left(\rho(t_{0})-t\right)(t-t_{2})\ \mbox{for } t\in [t_{2},\rho(t_{0})],\\
0\ \mbox{otherwise}.
\end{array}\right.
$$
Then,
$$
\eta(\rho(t))=\left\{
\begin{array}{l}
\left(\rho(t_{0})-\rho(t)\right)(\rho(t)-t_{2})\ \mbox{for } \rho(t)\in [t_{2},\rho(t_{0})],\\
0\ \mbox{otherwise}.
\end{array}\right.
$$
On the interval $[t_{2},\rho(t_{0})]$ the function $\eta(\rho(t))$ is greater than $0$. Then,
$$
\lim\limits_{T\to +\infty} \inf\limits_{T'\geq T}\int\limits_{a}^{T'}g(t)\eta^{\rho}(t)\nabla t
= \int\limits_{t_{2}}^{\rho(t_{0})} g(t)\eta^{\rho}(t)\nabla t>0,
$$
which is a contradiction. Suppose that $g(\rho(t_{0})) =0$.
Here two situations may occur:
(i) $g(t)=0$ on $[t_{3},\rho(t_{0})]$ for some $t_{3}<\rho(t_{0})$ or
(ii) for all $t_{3}<\rho(t_{0})$ there exists $t \in [t_{3},\rho(t_{0})]$
such that $g(t)\neq 0$. In case (i) $t_{3}<\rho(t_{0})<t_{0}$. Let us define
$$
\eta(t)=\left\{
\begin{array}{l}
g(t_{0}) \ \mbox{ for } t=\rho(t_{0}),\\
\varphi(t) \ \mbox{ for } t\in [t_{3},\rho(t_{0})[,\\
0\ \mbox{ otherwise},
\end{array}\right.
$$
for function $\varphi$ such that $\varphi \in C_{ld}$,
$\varphi(t_{3})=0$ and $\varphi(\rho(t_{0}))=g(t_{0})$. Then,
\begin{equation*}
\eta(\rho(t))=\left\{
\begin{array}{l}
g(t_{0})\ \mbox{ for } \rho(t)=\rho(t_{0}),\\
\varphi(\rho(t)) \mbox{ for } \rho(t)\in [t_{3},\rho(t_{0})[,\\
0\ \mbox{otherwise}.
\end{array}\right.
\end{equation*}
It follows from point 5 of Theorem~\ref{tw:intprop} that
\begin{equation*}
\begin{split}
\lim\limits_{T\to +\infty} \inf\limits_{T'\geq T}\int\limits_{a}^{T'}g(t)\eta^{\rho}(t)\nabla t
&=\int\limits_{t_{3}}^{t_{0}} g(t)\eta^{\rho}(t)\nabla t
=\int\limits_{\rho(t_{0})}^{t_{0}} g(t)\eta^{\rho}(t)\nabla t\\
&=\nu(t_{0})g(t_{0})\eta(\rho(t_{0}))=(t_{0}-\rho(t_{0}))g(t_{0})\eta(\rho(t_{0}))>0,
\end{split}
\end{equation*}
which is a contradiction.
In case (ii), $t_{3}<\rho (t_{0})<t_{0}$.
When $\rho(t_{0})$ is left-dense, then there exists a strictly increasing sequence
$S=\lbrace s_{k}: k\in\mathbb{N}\rbrace\subseteq\mathbb{T}$ such that
$\lim\limits_{k\to\infty} s_{k}=\rho(t_{0})$ and $g(s_{k})\neq 0$ for all $k\in\mathbb{N}$.
If there exists a left-dense $s_{k}$, then we have Case I with $t_{0}:=s_{k}$.
If all points of the sequence $S$  are left-scattered, then we have Case II
with $t_{0}:=s_{i}$, $i\in\mathbb{N}$. Since $\rho(t_{0})$ is a left-scattered point,
we are in the first situation of case II and we obtain a contradiction.
Therefore, we conclude that $g\equiv 0$ on $[a,+\infty[$.
\end{proof}

\begin{corollary}
\label{cor:Dubois-Reymond nabla}
Let $h\in C_{ld}(\mathbb{T}, \mathbb {R})$. Then,
\begin{equation}
\label{eq:cor}
\lim\limits_{T\to\infty} \inf\limits_{T'\geq T}
\int\limits_{a}^{T'} h(t)\eta^{\nabla}(t)\nabla t =0
\end{equation}
for all $\eta\in C_{ld}\left(\mathbb{T},\mathbb {R}\right)$
such that $\eta (a)=0$ if, and only if, $h(t)=c$,
$c\in\mathbb{R}$, on $[a, +\infty[$.
\end{corollary}

\begin{proof}
Using integration by parts (third item of Theorem~\ref{tw:intprop}),
\begin{equation*}
\int\limits_{a}^{T'}h(t) \eta^{\nabla}(t)\nabla t
=\left.h(t)\eta(t)\right|_{t=a}^{t=T'}
- \int\limits_{a}^{T'}h^{\nabla}(t)\eta^{\rho}(t)\nabla t
= h(T')\eta(T') - \int\limits_{a}^{T'}h^{\nabla}(t)\eta^{\rho}(t)\nabla t
\end{equation*}
holds for all $\eta\in C_{ld}(\mathbb{T}, \mathbb {R})$.
In particular, it holds for the subclass of $\eta$ with $\eta(T')=0$
and \eqref{eq:cor} is equivalent to
\begin{equation*}
\lim\limits_{T\to\infty} \inf\limits_{T'\geq T}
\int\limits_{a}^{T'}h^{\nabla}(t)\eta^{\rho}(t)\nabla t=0.
\end{equation*}
Using Lemma~\ref{lem:Dubois-Reymond rho}, we obtain $h^{\nabla}(t)=0$,
\textrm{i.e.}, $h(t)=c$, $c\in\mathbb{R}$, on $[a, +\infty[$.
\end{proof}

\begin{theorem}
\label{maintheorem}
Suppose that a weak maximizer to problem \eqref{eq:main result}
exists and is given by $x_{\star}$.
Let $p \in C_{ld}^{1}\left(\mathbb{T},\mathbb {R}^{n}\right)$
be such that $p(a)=0$. Define
$$
A(\varepsilon, T'):= \int\limits_{a}^{T'}\frac{L\left(t, x_{\star}^{\rho}(t)
+\varepsilon p^{\rho}(t),x_{\star}^{\nabla}(t)
+\varepsilon p^{\nabla}(t),z_{\star}(t,p)\right)
-L\left[x_{\star},z_{\star}\right](t)}{\varepsilon}\nabla t,
$$
where
\begin{equation*}
\begin{split}
z_{\star}(t,p)&=\int\limits_{a}^{t}
g\langle x_{\star} + \varepsilon p \rangle(\tau) \nabla \tau,\\
z_{\star}(t)&=\int\limits_{a}^{t} g \langle x_{\star}\rangle(\tau)\nabla \tau,
\end{split}
\end{equation*}
and
\begin{equation*}
\begin{split}
V(\varepsilon, T) &:= \inf\limits_{T'\geq T} \varepsilon A(\varepsilon,T'),\\
V(\varepsilon) &:=\lim\limits_{T\to\infty} V(\varepsilon, T).
\end{split}
\end{equation*}
Suppose that
\begin{enumerate}
\item $\lim\limits_{\varepsilon\to 0}\frac{V(\varepsilon, T)}{\varepsilon}$ exists for all $T$;
\item $\lim\limits_{T\to\infty}\frac{V(\varepsilon, T)}{\varepsilon}$ exists uniformly for $\varepsilon$;
\item for every $T'>a$, $T>a$, $\varepsilon\in\mathbb{R} \setminus \lbrace 0\rbrace$,
there exists a sequence $\left(A(\varepsilon, T'_{n})\right)_{n\in\mathbb{N}}$ such that
$\lim\limits_{n\to\infty} A(\varepsilon, T'_{n})
=\inf\limits_{T'\geq T} A(\varepsilon, T')$ uniformly for $\varepsilon$.
\end{enumerate}
Then, $x_{\star}$ satisfies the Euler--Lagrange system of $n$ equations
\begin{multline}
\label{eq:ELs:neq}
\lim\limits_{T\to\infty}\inf\limits_{T'\geq T}
\left\{ g_{x}\langle x\rangle(t)
\int\limits_{\rho(t)}^{T'}L_{z}[x,z](\tau)\nabla \tau
- \left(g_{v}\langle x\rangle(t)\int\limits_{\rho(t)}^{T'}
L_{z}[x,z](\tau)\nabla\tau\right)^{\nabla}\right\}\\
+ L_{x}[x,z](t)-L_{v}^{\nabla}[x,z](t)=0
\end{multline}
for all $t\in [a,+\infty[$ and the transversality condition
\begin{equation}
\label{eq:transCond}
\lim\limits_{T\to\infty}\inf\limits_{T'\geq T}
\left\{x(T') \cdot \left[L_{v}[x,z](T')
+ g_{v}\langle x\rangle(T') \nu(T') L_{z}[x,z](T')\right]\right\}=0.
\end{equation}
\end{theorem}

\begin{proof}
If $x_{\star}$ is optimal, in the sense of Definition~\ref{df:2},
then $V(\varepsilon)\leq 0$ for any $\varepsilon\in\mathbb{R}$.
Because $V(0)=0$, then $0$ is a maximizer of $V$. We prove that
$V$ is differentiable at $0$, thus $V'(0)=0$. Note that
\begin{equation*}
\begin{split}
0&=V'(0)=\lim\limits_{\varepsilon\to 0}\frac{V(\varepsilon)}{\varepsilon}
=\lim\limits_{\varepsilon\to 0}\lim\limits_{T\to\infty}\frac{V(\varepsilon, T)}{\varepsilon}
=\lim\limits_{T\to\infty}\lim\limits_{\varepsilon\to 0}\frac{V(\varepsilon,T)}{\varepsilon}\\
&=\lim\limits_{T\to\infty}\lim\limits_{\varepsilon\to 0}\inf\limits_{T' \geq T} A(\varepsilon, T')
=\lim\limits_{T\to\infty}\lim\limits_{\varepsilon\to 0}\lim\limits_{n\to\infty} A(\varepsilon, T^{'}_{n})\\
&=\lim\limits_{T\to\infty}\lim\limits_{n\to\infty}\lim\limits_{\varepsilon\to 0} A(\varepsilon, T^{'}_{n})
=\lim\limits_{T\to\infty}\inf\limits_{T'\geq T}\lim\limits_{\varepsilon\to 0} A(\varepsilon, T')\\
&=\lim\limits_{T\to\infty}\inf\limits_{T'\geq T}\lim\limits_{\varepsilon\to 0}
\int\limits_{a}^{T'}\frac{L\left(t, x_{\star}^{\rho}(t)+\varepsilon p^{\rho}(t),x_{\star}^{\nabla}(t)
+\varepsilon p^{\nabla}(t),z_{\star}(t,p)\right)
-L\left[x_{\star},z_{\star}\right](t)}{\varepsilon}\nabla t\\
&=\lim\limits_{T\to\infty}\inf\limits_{T'\geq T} \int\limits_{a}^{T'}\lim\limits_{\varepsilon\to 0}
\frac{L\left(t, x_{\star}^{\rho}(t)+\varepsilon p^{\rho}(t),x_{\star}^{\nabla}(t)
+\varepsilon p^{\nabla}(t),z_{\star}(t,p)\right)
-L\left[x_{\star},z_{\star}\right](t)}{\varepsilon}\nabla t,
\end{split}
\end{equation*}
that is,
\begin{multline}
\label{eq:rownanie do zmieniania}
\lim\limits_{T\to\infty}\inf\limits_{T'\geq T}\int\limits_{a}^{T'} \Biggl[
L_{x}[x_{\star},z_{\star}](t) \cdot p^{\rho}(t)
+ L_{v}[x_{\star},z_{\star}](t) \cdot p^{\nabla}(t)\\
+ L_{z}[x_{\star},z_{\star}](t) \int\limits_{a}^{t}
\left(g_{x}\langle x_{\star}\rangle(\tau)
\cdot p^{\rho}(\tau)+g_{v}\langle x_{\star}\rangle(\tau)
\cdot p^{\nabla}(\tau)\right)\nabla \tau \Biggr]\nabla t = 0.
\end{multline}
Using the integration by parts formula given
by point 3 of Theorem~\ref{tw:intprop}, we obtain:
\begin{equation*}
\begin{split}
\int\limits_{a}^{T'}L_{v}[x_{\star},z_{\star}](t) \cdot p^{\nabla}(t)\nabla t
&= \left.L_{v}[x_{\star},z_{\star}](t) \cdot p(t)\right|_{t=a}^{t=T'}-\int\limits_{a}^{T'}
L_{v}^{\nabla}[x_{\star},z_{\star}](t) \cdot p^{\rho}(t)\nabla t\\
&= L_{v}[x_{\star},z_{\star}](T') \cdot p(T')
-\int\limits_{a}^{T'}L_{v}^{\nabla}[x_{\star},z_{\star}](t) \cdot p^{\rho}(t)\nabla t.
\end{split}
\end{equation*}
Next, we consider the last part of equation \eqref{eq:rownanie do zmieniania}.
First we use the third nabla differentiation formula of Theorem~\ref{tw:differprop}:
\begin{equation*}
\begin{split}
&\left[\int\limits_{t}^{T'}L_{z}[x_{\star},z_{\star}](\tau)\nabla \tau
\int\limits_{a}^{t}\left(g_{x}\langle x_{\star}\rangle(\tau) \cdot p^{\rho}(\tau)
+g_{v}\langle x_{\star}\rangle(\tau) \cdot p^{\nabla}(\tau)\right)\nabla \tau\right]^{\nabla}\\
&=\left(\int\limits_{t}^{T'} L_{z}[x_{\star},z_{\star}](\tau)\nabla \tau\right)^{\nabla}
\int\limits_{a}^{t}\left(g_{x}\langle x_{\star}\rangle(\tau)
\cdot p^{\rho}(\tau)+g_{v}\langle x_{\star}\rangle(\tau)
\cdot p^{\nabla}(\tau)\right)\nabla\tau\\
&\qquad +\left(\int\limits_{\rho(t)}^{T'}L_{z}[x_{\star},z_{\star}](\tau)\nabla \tau\right)
\left(\int\limits_{a}^{t}\left(g_{x}\langle x_{\star}\rangle(\tau)
\cdot p^{\rho}(\tau)+g_{v}\langle x_{\star}\rangle(\tau)
\cdot p^{\nabla}(\tau)\right)\nabla \tau\right)^\nabla\\
&=-L_{z}[x_{\star},z_{\star}](t)\int\limits_{a}^{t}\left(g_{x}\langle x_{\star}\rangle(\tau)
\cdot p^{\rho}(\tau)+g_{v}\langle x_{\star}\rangle(\tau)
\cdot p^{\nabla}(\tau)\right)\nabla\tau\\
&\qquad +\left(\int\limits_{\rho(t)}^{T'}L_{z}[x_{\star},z_{\star}](\tau)\nabla \tau\right)
\left(g_{x}\langle x_{\star}\rangle(t)
\cdot p^{\rho}(t)+g_{v}\langle x_{\star}\rangle(t) \cdot p^{\nabla}(t)\right).
\end{split}
\end{equation*}
Integrating both sides from $t = a$ to $t=T'$,
\begin{equation*}
\begin{split}
\int\limits_{a}^{T'}&\left[\int\limits_{t}^{T'}L_{z}[x_{\star},z_{\star}](\tau)\nabla \tau\int\limits_{a}^{t}
\left(g_{x}\langle x_{\star}\rangle(\tau)
\cdot p^{\rho}(\tau)+g_{v}\langle x_{\star}\rangle(\tau)
\cdot p^{\nabla}(\tau)\right)\nabla \tau\right]^{\nabla}\nabla t\\
&=-\int\limits_{a}^{T'}\left[L_{z}[x_{\star},z_{\star}](t)
\int\limits_{a}^{t}\left(g_{x}\langle x_{\star}\rangle(\tau) \cdot p^{\rho}(\tau)
+g_{v}\langle x_{\star}\rangle(\tau) \cdot p^{\nabla}(\tau)\right)\nabla\tau\right]\nabla t\\
&\qquad +\int\limits_{a}^{T'}\left[\int\limits_{\rho(t)}^{T'}
L_{z}[x_{\star},z_{\star}](\tau)\nabla\tau \left(g_{x}\langle x_{\star}\rangle(t)
\cdot p^{\rho}(t)+g_{v}\langle x_{\star}\rangle(t) \cdot p^{\nabla}(t)\right)\right]\nabla t.
\end{split}
\end{equation*}
The left-hand side of above equation is zero,
\begin{multline*}
\int\limits_{a}^{T'}\left[\int\limits_{t}^{T'}L_{z}[x_{\star},z_{\star}](\tau)\nabla\tau
\int\limits_{a}^{t} \left(g_{x}\langle x_{\star}\rangle(\tau) \cdot p^{\rho}(\tau)
+g_{v}\langle x_{\star}\rangle(\tau) \cdot p^{\nabla}(\tau)\right)\nabla \tau\right]^\nabla\nabla t\\
=\left.\int\limits_{t}^{T'}L_{z}[x_{\star},z_{\star}](\tau)\nabla\tau
\int\limits_{a}^{t}\left(g_{x}\langle x_{\star}\rangle(\tau) \cdot p^{\rho}(\tau)
+g_{v}\langle x_{\star}\rangle(\tau) \cdot p^{\nabla}(\tau)\right)\nabla\tau\right|^{t=T'}_{t=a}=0,
\end{multline*}
and, therefore,
\begin{equation}
\label{rownanie z g}
\begin{split}
\int\limits_{a}^{T'}&\left[L_{z}[x_{\star},z_{\star}](t)
\int\limits_{a}^{t}\left(g_{x}\langle x_{\star}\rangle(\tau) \cdot p^{\rho}(\tau)
+g_{v}\langle x_{\star}\rangle(\tau)
\cdot p^{\nabla}(\tau)\right)\nabla\tau\right]\nabla t\\
&=\int\limits_{a}^{T'}\left[\int\limits_{\rho(t)}^{T'}
L_{z}[x_{\star},z_{\star}](\tau)\nabla\tau
\left(g_{x}\langle x_{\star}\rangle(t)
\cdot p^{\rho}(t)+g_{v}\langle x_{\star}\rangle(t)
\cdot p^{\nabla}(t)\right)\right]\nabla t\\
&=\int\limits_{a}^{T'}\left[g_{x}\langle x_{\star}\rangle(t)
\cdot p^{\rho}(t) \int\limits_{\rho(t)}^{T'}
L_{z}[x_{\star},z_{\star}](\tau)\nabla \tau\right]\nabla t\\
&\qquad +\int\limits_{a}^{T'}\left[p^{\nabla}(t) \cdot g_{v}\langle x_{\star}\rangle(t)
\int\limits_{\rho(t)}^{T'}L_{z}[x_{\star},z_{\star}](\tau)\nabla\tau\right]\nabla t.
\end{split}
\end{equation}
Using point~3 of Theorem~\ref{tw:intprop} and the fact that $p(a)=0$,
\begin{multline*}
\int\limits_{a}^{T'} \left[p^{\nabla}(t) \cdot g_{v}\langle x_{\star}\rangle(t)
\int\limits_{\rho(t)}^{T'}
L_{z}[x_{\star},z_{\star}](\tau)\nabla \tau \right]\nabla t\\
=p(T') \cdot g_{v}\langle x_{\star}\rangle(T')
\int\limits_{\rho(T')}^{T'}L_{z}[x_{\star},z_{\star}](\tau)\nabla\tau\\
- \int\limits_{a}^{T'}\left(g_{v}\langle x_{\star}\rangle(t) \int\limits_{\rho(t)}^{T'}
L_{z}[x_{\star},z_{\star}](\tau)\nabla \tau\right)^{\nabla} \cdot p^{\rho}(t)\nabla t.
\end{multline*}
Then, from \eqref{eq:rownanie do zmieniania},
\begin{multline*}
\lim\limits_{T\to\infty}\inf\limits_{T'\geq T}\int\limits_{a}^{T'}
L_{x}[x_{\star},z_{\star}](t) \cdot p^{\rho}(t)\nabla t
+ L_{v}[x_{\star},z_{\star}](T') \cdot p(T')
- \int\limits_{a}^{T'} L_{v}^{\nabla}[x_{\star},z_{\star}](t) \cdot p^{\rho}(t)\nabla t\\
+ \int\limits_{a}^{T'} \left(\int\limits_{\rho(t)}^{T'}
L_{z}[x_{\star},z_{\star}](\tau)\nabla\tau \left(g_{x}\langle x_{\star}\rangle(t)
\cdot p^{\rho}(t) + g_{v}\langle x_{\star}\rangle(t)
\cdot p^{\nabla}(t)\right)\right)\nabla t\\
\end{multline*}
\begin{equation}
\label{eq:wynik rho}
\begin{split}
&=\lim\limits_{T\to\infty}\inf\limits_{T'\geq T}\int\limits_{a}^{T'}
L_{x}[x_{\star},z_{\star}](t) \cdot p^{\rho}(t)\nabla t
+ L_{v}[x_{\star},z_{\star}](T') \cdot p(T')\\
&\qquad -\int\limits_{a}^{T'}\left\{ L_{v}^{\nabla}[x_{\star},z_{\star}](t) \cdot p^{\rho}(t)
+ \int\limits_{\rho(t)}^{T'} L_{z}[x_{\star},z_{\star}](\tau)\nabla\tau
g_{x}\langle x_{\star}\rangle(t) \cdot p^{\rho}(t)\right\}\nabla t\\
&\qquad + g_{v}\langle x_{\star}\rangle(T')
\int\limits_{\rho(T')}^{T'}L_{z}[x_{\star},z_{\star}](\tau)\nabla\tau  \cdot p(T')\\
&\qquad -\int\limits_{a}^{T'}\left(g_{v}\langle x_{\star}\rangle(t)
\int\limits_{\rho(t)}^{T'}
L_{z}[x_{\star},z_{\star}](\tau)\nabla \tau \right)^{\nabla} \cdot p^{\rho}(t)\nabla t\\
&=\lim\limits_{T\to\infty}\inf\limits_{T'\geq T}\left\{\int\limits_{a}^{T'} p^{\rho}(t)
\cdot \Biggl[L_{x}[x_{\star},z_{\star}](t) - L_{v}^{\nabla}[x_{\star},z_{\star}](t)\right.\\
&\qquad \left.+ g_{x}\langle x_{\star}\rangle(t)
\int\limits_{\rho(t)}^{T'} L_{z}[x_{\star},z_{\star}](\tau)\nabla\tau
-\left(g_{v}\langle x_{\star}\rangle(t)
\int\limits_{\rho(t)}^{T'}L_{z}[x_{\star},z_{\star}](\tau)
\nabla \tau \right)^{\nabla}\Biggr]\nabla t\right.\\
&\qquad \left.+L_{v}[x_{\star},z_{\star}](T') \cdot p(T')
+\left(g_{v}\langle x_{\star}\rangle(T')
\int\limits_{\rho(T')}^{T'}L_{z}[x_{\star},z_{\star}](\tau)\nabla\tau\right) \cdot p(T')\right\}=0.
\end{split}
\end{equation}
We know that equation \eqref{eq:wynik rho} holds for all $p\in C^{1}_{ld}$ such that $p(a)=0$,
then, in particular, it also holds for the subclass of $p$ with $p(T')=0$. Therefore,
\begin{multline*}
\lim\limits_{T\to\infty}\inf\limits_{T'\geq T}\int\limits_{a}^{T'}
p^{\rho}(t) \cdot \left[L_{x}[x_{\star},z_{\star}](t)
-L_{v}^{\nabla}[x_{\star},z_{\star}](t) + g_{x}\langle x_{\star}\rangle(t)
\int\limits_{\rho(t)}^{T'}
L_{z}[x_{\star},z_{\star}](\tau)\nabla\tau \right.\\
\left.-\left(g_{v}\langle x_{\star}\rangle(t) \int\limits_{\rho(t)}^{T'}
L_{z}[x_{\star},z_{\star}](\tau)\nabla \tau \right)^{\nabla}\right]\nabla t=0.
\end{multline*}
Choosing $p=\left(p_{1},\ldots,p_{n}\right)$
such that $p_{2}\equiv\cdots\equiv p_{n}\equiv 0$,
\begin{multline*}
\lim\limits_{T\to\infty}\inf\limits_{T'\geq T}\int\limits_{a}^{T'}
p^{\rho}_{1}(t)\left[ L_{x_{1}}[x_{\star},z_{\star}](t)
+ g_{x_{1}}\langle x_{\star}\rangle(t)
\int\limits_{\rho(t)}^{T'} L_{z}[x_{\star},z_{\star}](\tau)\nabla\tau\right.\\
\left. -L_{v_{1}}^{\nabla}[x_{\star},z_{\star}](t)
-\left(g_{v_{1}}\langle x_{\star}\rangle(t) \int\limits_{\rho(t)}^{T'}
L_{z}[x_{\star},z_{\star}](\tau)\nabla \tau \right)^{\nabla}\right]\nabla t=0.
\end{multline*}
Using Lemma~\ref{lem:Dubois-Reymond rho},
\begin{multline*}
g_{x_{1}}\langle x_{\star}\rangle(t)
\int\limits_{\rho(t)}^{T'} L_{z}[x_{\star},z_{\star}](\tau)\nabla\tau
-\left(g_{v_{1}}\langle x_{\star}\rangle(t) \int\limits_{\rho(t)}^{T'}
L_{z}[x_{\star},z_{\star}](\tau)\nabla \tau \right)^{\nabla}\\
+ L_{x_{1}}[x_{\star},z_{\star}](t) - L_{v_{1}}^{\nabla}[x_{\star},z_{\star}](t) =0
\end{multline*}
for all $t\in [a,+\infty[$ and all $T' \ge t$. We can do the same for other coordinates.
For all $i=1,\ldots,n$ we obtain the equation
\begin{multline*}
g_{x_{i}}\langle x_{\star}\rangle(t)
\int\limits_{\rho(t)}^{T'} L_{z}[x_{\star},z_{\star}](\tau)\nabla\tau
-\left(g_{v_{i}}\langle x_{\star}\rangle(t) \int\limits_{\rho(t)}^{T'}
L_{z}[x_{\star},z_{\star}](\tau)\nabla \tau \right)^{\nabla}\\
+ L_{x_{i}}[x_{\star},z_{\star}](t)
-L_{v_{i}}^{\nabla}[x_{\star},z_{\star}](t)=0
\end{multline*}
for all $t\in [a,+\infty[$ and all $T' \ge t$.
These $n$ conditions can be written in vector form as
\begin{multline}
\label{eq:wynik rho po tw}
g_{x}\langle x_{\star}\rangle(t)
\int\limits_{\rho(t)}^{T'} L_{z}[x_{\star},z_{\star}](\tau)\nabla\tau
-\left(g_{v}\langle x_{\star}\rangle(t) \int\limits_{\rho(t)}^{T'}
L_{z}[x_{\star},z_{\star}](\tau)\nabla \tau\right)^{\nabla}\\
+L_{x}[x_{\star},z_{\star}](t) -L_{v}^{\nabla}[x_{\star},z_{\star}](t)=0
\end{multline}
for all $t\in [a,+\infty[$ and all $T' \ge t$, which implies
the Euler--Lagrange system of $n$ equations \eqref{eq:ELs:neq}.
From the system of equations \eqref{eq:wynik rho po tw}
and equation \eqref{eq:wynik rho}, we conclude that
\begin{equation}
\label{tran}
\lim\limits_{T\to\infty}\inf\limits_{T'\geq T}
\left\{ \left( L_{v}[x_{\star},z_{\star}](T')
+g_{v}\langle x_{\star}\rangle(T')
\int\limits_{\rho(T')}^{T'}L_{z}[x_{\star},z_{\star}](\tau)\nabla\tau\right)
\cdot p(T')\right\} =0.
\end{equation}
Next, we define a special curve $p$: for all $t\in [a,\infty[$
\begin{equation}
\label{curve p}
p(t)=\alpha(t)x_{\star}(t) \, ,
\end{equation}
where $\alpha:[a,\infty[\rightarrow\mathbb{R}$ is a $C_{ld}^{1}$
function satisfying $\alpha(a)=0$ and for which there exists
$T_{0}\in\mathbb{T}$ such that $\alpha(t)=\beta\in\mathbb{R} \setminus \lbrace 0\rbrace$
for all $t>T_{0}$. Substituting $p(T') = \alpha(T')x_{\star}(T')$
into \eqref{tran}, we conclude that
\begin{equation*}
\lim\limits_{T\to\infty}\inf\limits_{T'\geq T}\left\{
L_{v}[x_{\star},z_{\star}](T') \cdot \beta x_{\star}(T')
+g_{v}\langle x_{\star}\rangle(T')
\int\limits_{\rho(T')}^{T'}L_{z}[x_{\star},z_{\star}](\tau)\nabla\tau
\cdot \beta x_{\star}(T')\right\}
\end{equation*}
vanishes and, therefore,
$$
\lim\limits_{T\to\infty}\inf\limits_{T'\geq T}
\left\{x_{\star}(T') \cdot \left[L_{v}[x_{\star},z_{\star}](T')
+g_{v}\langle x_{\star}\rangle(T')\int\limits_{\rho(T')}^{T'}
L_{z}[x_{\star},z_{\star}](\tau)\nabla\tau\right]\right\} =0.
$$
From item~5 of Theorem~\ref{tw:intprop},
$x_{\star}$ satisfies the transversality
condition \eqref{eq:transCond}.
\end{proof}

In contrast with Theorem~\ref{maintheorem},
the following theorem is proved by manipulating equation
\eqref{eq:rownanie do zmieniania} differently:
using integration by parts and nabla differentiation formulas,
we transform the items which consist of $p^{\rho}$ into $p^\nabla$.
Thanks to that, we apply Corollary~\ref{cor:Dubois-Reymond nabla}
instead of Lemma~\ref{lem:Dubois-Reymond rho}
to obtain the intended conclusions.

\begin{theorem}
\label{secondmaintheorem}
Under assumptions of Theorem~\ref{maintheorem},
the Euler--Lagrange system of $n$ equations
\begin{multline}
\label{wynik nabla po tw}
\lim\limits_{T\to\infty}\inf\limits_{T'\geq T}\left\{
\int\limits_{t}^{T'} g_{x}\langle x_{\star}\rangle(\tau)
\int\limits_{\rho(\tau)}^{T'}
L_{z}[x_{\star},z_{\star}](s)\nabla s \nabla\tau
+g_{v}\langle x_{\star}\rangle(t) \int\limits_{\rho(t)}^{T'}
L_{z}[x_{\star},z_{\star}](\tau)\nabla \tau \right\}\\
+L_{v}[x_{\star},z_{\star}](t)
-\int\limits_{a}^{t} L_{x}[x_{\star},z_{\star}](\tau)\nabla\tau = c
\end{multline}
holds for all $t\in [a,\infty[$, $c\in\mathbb{R}^n$,
together with the transversality condition
\begin{equation}
\label{eq:transCond:2ndT}
\lim\limits_{T\to\infty}\inf\limits_{T'\geq T} \left\{
x_{\star}(T') \cdot
\int\limits_{a}^{T'} L_{x}[x_{\star},z_{\star}](\tau)\nabla\tau \right\} = 0.
\end{equation}
\end{theorem}

\begin{proof}
We use the necessary optimality condition \eqref{eq:rownanie do zmieniania}
found in the proof of Theorem~\ref{maintheorem}.
Using point~3 of Theorem~\ref{tw:differprop},
\begin{equation*}
\begin{split}
\Biggl[p(t) & \cdot \int\limits_{a}^{t}L_{x}[x_{\star},z_{\star}](\tau)\nabla \tau \Biggr]^{\nabla}\\
&= p^{\nabla}(t) \cdot \int\limits_{a}^{t} L_{x}[x_{\star},z_{\star}](\tau)\nabla \tau
+p^{\rho}(t) \cdot
\left[\int\limits_{a}^{t} L_{x}[x_{\star},z_{\star}](\tau)\nabla\tau\right]^{\nabla}\\
&=p^{\nabla}(t) \cdot \int\limits_{a}^{t}L_{x}[x_{\star},z_{\star}](\tau)\nabla\tau
+p^{\rho}(t) \cdot L_{x}[x_{\star},z_{\star}](t).
\end{split}
\end{equation*}
Then, integrating both sides from $t = a$ to $t = T'$,
\begin{multline*}
\int\limits_{a}^{T'}\left[p(t) \cdot
\int\limits_{a}^{t}
L_{x}[x_{\star},z_{\star}](\tau)\nabla \tau \right]^{\nabla}\nabla t\\
=\int\limits_{a}^{T'} \left(p^{\nabla}(t) \cdot \int\limits_{a}^{t}
L_{x}[x_{\star},z_{\star}](\tau)\nabla \tau\right) \nabla t
+\int\limits_{a}^{T'} p^{\rho}(t) \cdot L_{x}[x_{\star},z_{\star}](t)\nabla t.
\end{multline*}
Therefore,
\begin{multline*}
\left. p(t) \cdot \int\limits_{a}^{t}
L_{x}[x_{\star},z_{\star}](\tau)\nabla \tau \right|_{t=a}^{t=T'}\\
= \int\limits_{a}^{T'} \left(p^{\nabla}(t) \cdot \int\limits_{a}^{t}
L_{x}[x_{\star},z_{\star}](\tau)\nabla \tau\right)\nabla t
+\int\limits_{a}^{T'} p^{\rho}(t) \cdot L_{x}[x_{\star},z_{\star}](t)\nabla t.
\end{multline*}
Since $p(a)=0$, we obtain that
$\displaystyle p(T') \cdot \int\limits_{a}^{T'} L_{x}[x_{\star},z_{\star}](\tau)\nabla\tau$
is equal to
\begin{equation*}
\int\limits_{a}^{T'} p^{\nabla}(t) \cdot
\left(\int\limits_{a}^{t} L_{x}[x_{\star},z_{\star}](\tau)\nabla \tau\right)\nabla t
+\int\limits_{a}^{T'} p^{\rho}(t) \cdot L_{x}[x_{\star},z_{\star}](t)\nabla t,
\end{equation*}
that is,
\begin{multline*}
\int\limits_{a}^{T'} p^{\rho}(t) \cdot L_{x}[x_{\star},z_{\star}](t)\nabla t\\
=-\int\limits_{a}^{T'} p^{\nabla}(t) \cdot
\left(\int\limits_{a}^{t} L_{x}[x_{\star},z_{\star}](\tau)\nabla \tau\right)\nabla t
+ p(T') \cdot \int\limits_{a}^{T'} L_{x}[x_{\star},z_{\star}](\tau)\nabla\tau.
\end{multline*}
Making the same calculations as in the proof of Theorem~\ref{maintheorem},
we obtain \eqref{rownanie z g}. Using again
point~3 of Theorem~\ref{tw:differprop},
\begin{equation*}
\begin{split}
\Biggl[p(t) \cdot & \int\limits_{t}^{T'}\left(g_{x}\langle x_{\star}\rangle(\tau)
\int\limits_{\rho(\tau)}^{T'}
L_{z}[x_{\star},z_{\star}](s)\nabla s\right)\nabla\tau\Biggr]^{\nabla}\\
&=p^{\nabla}(t) \cdot \int\limits_{t}^{T'}\left(g_{x}\langle x_{\star}\rangle(\tau)
\int\limits_{\rho(\tau)}^{T'}L_{z}[x_{\star},z_{\star}](s)\nabla s\right)\nabla\tau\\
&\qquad + p^{\rho}(t) \cdot \left[\int\limits_{t}^{T'}\left(
g_{x}\langle x_{\star}\rangle(\tau)\int\limits_{\rho(\tau)}^{T'}
L_{z}[x_{\star},z_{\star}](s)\nabla s\right)\nabla\tau\right]^{\nabla}\\
&= p^{\nabla}(t) \cdot \int\limits_{t}^{T'}\left(g_{x}\langle x_{\star}\rangle(\tau)
\int\limits_{\rho(\tau)}^{T'}L_{z}[x_{\star},z_{\star}](s)\nabla s\right)\nabla\tau\\
& \qquad
- p^{\rho}(t) \cdot g_{x}\langle x_{\star}\rangle(t)
\int\limits_{\rho(t)}^{T'}L_{z}[x_{\star},z_{\star}](\tau)\nabla\tau.
\end{split}
\end{equation*}
Integrating both sides from $t = a$ to $t = T'$,
and because of point~2 of Theorem~\ref{tw:intprop} and $p(a)=0$,
\begin{multline*}
\int\limits_{a}^{T'}\left[p(t) \cdot
\int\limits_{t}^{T'}\left(g_{x}\langle x_{\star}\rangle(\tau)
\int\limits_{\rho(\tau)}^{T'}
L_{z}[x_{\star},z_{\star}](s)\nabla s\right)\nabla\tau\right]^{\nabla}\nabla t\\
=\left. p(t) \cdot \int\limits_{t}^{T'}\left(g_{x}\langle x_{\star}\rangle(\tau)
\int\limits_{\rho(\tau)}^{T'}
L_{z}[x_{\star},z_{\star}](s)\nabla s\right)\nabla \tau\right|^{t=T'}_{t=a}=0.
\end{multline*}
Then,
\begin{multline*}
\int\limits_{a}^{T'} p^{\rho}(t) \cdot
\left(g_{x}\langle x_{\star}\rangle(t)
\int\limits_{\rho(t)}^{T'} L_{z}[x_{\star},z_{\star}](s)\nabla s\right)\nabla t\\
=\int\limits_{a}^{T'} p^{\nabla}(t) \cdot \left[
\int\limits_{t}^{T'}\left(g_{x}\langle x_{\star}\rangle(\tau)
\int\limits_{\rho(\tau)}^{T'}
L_{z}[x_{\star},z_{\star}](s)\nabla s\right)\nabla\tau\right] \nabla t.
\end{multline*}
From \eqref{rownanie z g} and previous relations,
we write \eqref{eq:rownanie do zmieniania} in the following way:
\begin{equation}
\label{wynik nabla}
\begin{split}
&\lim\limits_{T\to\infty}\inf\limits_{T'\geq T}\left\{
-\int\limits_{a}^{T'} p^{\nabla}(t) \cdot \int\limits_{a}^{t}
L_{x}[x_{\star},z_{\star}](\tau)\nabla\tau\nabla t
+ p(T') \cdot \int\limits_{a}^{T'}
L_{x}[x_{\star},z_{\star}](\tau)\nabla\tau \right.\\
&\ \left.+\int\limits_{a}^{T'}
L_{v}[x_{\star},z_{\star}](t) \cdot p^{\nabla}(t)\nabla t
+\int\limits_{a}^{T'} p^{\nabla}(t) \cdot \int\limits_{t}^{T'}\left(
g_{x}\langle x_{\star}\rangle(\tau)\int\limits_{\rho(\tau)}^{T'}
L_{z}[x_{\star},z_{\star}](s)\nabla s\right)\nabla\tau\nabla t\right.\\
&\ \left.+\int\limits_{a}^{T'}g_{v}\langle x_{\star}\rangle(t)
\cdot \left(p^{\nabla}(t)\int\limits_{\rho(t)}^{T'}
L_{z}[x_{\star},z_{\star}](\tau)\nabla \tau\right)\nabla t\right\}\\
&=\lim\limits_{T\to\infty}\inf\limits_{T'\geq T}\left\{
\int\limits_{a}^{T'} p^{\nabla}(t) \cdot \right.\left[
\int\limits_{a}^{t}- L_{x}[x_{\star},z_{\star}](\tau)\nabla\tau
+ L_{v}[x_{\star},z_{\star}](t)\right.\\
&\qquad \left.+\int\limits_{t}^{T'}\left(g_{x}\langle x_{\star}\rangle(\tau)
\int\limits_{\rho(\tau)}^{T'}
L_{z}[x_{\star},z_{\star}](s)\nabla s\right)\nabla\tau\right.\\
&\qquad \left.
+g_{v}\langle x_{\star}\rangle(t)
\int\limits_{\rho(t)}^{T'}L_{z}[x_{\star},z_{\star}](\tau)\nabla \tau
\right]\nabla t
\left. + p(T') \cdot
\int\limits_{a}^{T'} L_{x}[x_{\star},z_{\star}](\tau)\nabla\tau
\right\}=0.
\end{split}
\end{equation}
Because \eqref{wynik nabla} holds for all $p\in C_{ld}$
with $p(a)=0$, in particular it also holds
in the subclass of functions $p\in C_{ld}$
with $p(a)= p(T')=0$. Let $i \in \{1,\ldots,n\}$.
Choosing $p=\left(p_{1},\ldots,p_{n}\right)$
such that all $p_{j}\equiv 0$, $j \ne i$,
and $p_{i} \in C_{ld}$ with $p_i(a)= p_i(T')=0$,
we conclude that
\begin{multline*}
\lim\limits_{T\to\infty}\inf\limits_{T'\geq T}
\int\limits_{a}^{T'} p^{\nabla}_{i}(t) \left\{
\int\limits_{a}^{t}-L_{x_{i}}[x_{\star},z_{\star}](\tau)\nabla\tau
+L_{v_{i}}[x_{\star},z_{\star}](t)\right.\\
\left.+\int\limits_{t}^{T'}
g_{x_{i}}\langle x_{\star}\rangle(\tau)
\int\limits_{\rho(\tau)}^{T'}
L_{z}[x_{\star},z_{\star}](s)\nabla s
\nabla\tau
+ g_{v_{i}}\langle x_{\star}\rangle(t) \int\limits_{\rho(t)}^{T'}
L_{z}[x_{\star},z_{\star}](\tau)\nabla \tau\right\}\nabla t = 0.
\end{multline*}
From Corollary~\ref{cor:Dubois-Reymond nabla}
we obtain the equations
\begin{multline}
\label{eq:lmlm}
L_{v_{i}}[x_{\star},z_{\star}](t)
- \int\limits_{a}^{t}
L_{x_{i}}[x_{\star},z_{\star}](\tau)\nabla\tau
+\int\limits_{t}^{T'}\left(g_{x_{i}}\langle x_{\star}\rangle(\tau)
\int\limits_{\rho(\tau)}^{T'}
L_{z}[x_{\star},z_{\star}](s)\nabla s\right)\nabla\tau\\
+g_{v_{i}}\langle x_{\star}\rangle(t) \int\limits_{\rho(t)}^{T'}
L_{z}[x_{\star},z_{\star}](\tau)\nabla \tau =c_i,
\end{multline}
$c_i\in\mathbb{R}$, $i=1,\ldots,n$,
for all $t\in [a,+\infty[$ and all $T' \ge t$.
These $n$ conditions imply the
Euler--Lagrange system of equations \eqref{wynik nabla po tw}.
From \eqref{wynik nabla} and \eqref{eq:lmlm}, we conclude that
\begin{equation}
\label{tran dwa}
\lim\limits_{T\to\infty}\inf\limits_{T'\geq T}
\left\{ p(T') \cdot
\int\limits_{a}^{T'} L_{x}[x_{\star},z_{\star}](\tau)\nabla\tau \right\}= 0.
\end{equation}
Using the special curve $p$ defined by \eqref{curve p},
we obtain from equation \eqref{tran dwa} that
$$
\lim\limits_{T\to\infty}\inf\limits_{T'\geq T}
\left\{ \beta x_{\star}(T') \cdot \int\limits_{a}^{T'}
L_{x}[x_{\star},z_{\star}](\tau)\nabla\tau\right\} = 0.
$$
Therefore, $x_{\star}$ satisfies the transversality
condition \eqref{eq:transCond:2ndT}.
\end{proof}

% -------------------------------------

\section*{Acknowledgements}

Work supported by {\it FEDER} funds through
{\it COMPETE} --- Operational Programme Factors of Competitiveness
(``Programa Operacional Factores de Competitividade'')
and by Portuguese funds through the
{\it Center for Research and Development
in Mathematics and Applications} (University of Aveiro)
and the Portuguese Foundation for Science and Technology
(``FCT --- Funda\c{c}\~{a}o para a Ci\^{e}ncia e a Tecnologia''),
within project PEst-C/MAT/UI4106/2011
with COMPETE number FCOMP-01-0124-FEDER-022690.
Dryl was also supported by FCT through the Ph.D. fellowship
SFRH/BD/51163/2010; Torres by FCT through the project PTDC/MAT/113470/2009.
The authors are grateful to N. Martins for helpful suggestions
and for a detailed reading of a preliminary version of the paper.

% -------------------------------------

% -------------------------------------

\medskip

Received December 2011; revised October 2012; accepted December 2012.

\medskip

% -------------------------------------

\end{document}